 \documentclass[12pt,a4paper]{amsart}
\usepackage{times,amsfonts,amsmath,amstext,amsbsy,
  amsopn,amsthm,upref,a4} 
\usepackage[T1]{fontenc}

\newtheorem{theorem}{Theorem}[section]
\newtheorem{lemma}[theorem]{Lemma}
\newtheorem{addendum}[theorem]{Addendum}
\newtheorem{corollary}[theorem]{Corollary}

\newtheorem{rmk}[theorem]{Remark}

\theoremstyle{definition}
\newtheorem{definition}[theorem]{Definition}

\newcommand{\field}[1]{\mathbb{#1}}
\newcommand{\R}{\field{R}}
\newcommand{\N}{\field{N}}
\newcommand{\C}{\field{C}}

\newcommand{\Q}{\field{Q}}
\newcommand{\T}{\field{T}}

\newcommand{\Cal}{\mathcal}

\newcommand{\pref}[1]{(\ref{#1})}

\title[ Existence of  a Periodic Orbit for Billiards in Polygons]{ Existence of  a Periodic Orbit \\
for Billiards in Polygons}

\author{Giovanni Forni}
\address{Department of Mathematics\\
  University of Maryland\\
  College Park, MD  20742-4015, U.S.A.}

\email{gforni@math.umd.edu}

\thanks{ }

\date{\today}

\keywords{Polygonal billiards, Flat surfaces, Geodesic flows}

\subjclass{ }

\begin{document}

\begin{abstract}
  {We prove that the billiard flow in any finite polygon has at least one periodic orbit.
  The proof by contradiction is based on a fundamental result on the dynamics of the billiard flow 
  by Galperin, Kr\"uger and Troubetzkoy, on  the geometry of a one-parameter scaling of the natural Riemannian metric on the unit tangent bundle, and 
  on  the topology of the skeleton or cut-locus of the scaled metrics.}
\end{abstract}

\maketitle

\section{Introduction}

In this paper we present a non-constructive solution to the problem of existence of a periodic orbit in a polygonal billiard or more generally for the geodesic
flow on a flat compact surface with conical singularities:

\begin{theorem}
\label{thm:A}
Let $\Cal P\subset \R^2$ be any finite bounded polygon. The billiard flow in $\Cal P$ has 
at least one (regular) periodic orbit.
\end{theorem}

This theorem is obtained as a corollary of the following:

\begin{theorem}
\label{thm:B}
Let $(M,r)$ be any closed flat surface with a finite set of conical 
singularities. The geodesic flow on $(M,r)$ has at least one (regular) periodic orbit.
\end{theorem}

\noindent  The periodic orbit problem for billiards in polygons was included by A.~Katok in a list of
{\it Five Most Resistant Problems in Dynamics} (Problem 3 (ii) ) \cite{Ka04}.

For acute triangles the existence of a periodic orbit follows by a variational argument
going back to G.~Fagnano \cite{Fa} in 1775, but it is only with the work of 
R.~E.~Schwartz \cite{Sc08} that a proof of existence, constructive, has been given for obtuse triangles with one angle up to 100 degrees, 
a result later improved up to 112.3 degrees by G.~Tokarsky, J. Garber, B. Marinov, and K.~Moore ~\cite{TGMM}. 
In the special case of right triangles results on density of periodic 
orbits are known \cite{Trb05}. We refer the reader to the survey \cite{Sc_ICM} for results
and open problems on billiard dynamics, in particular in polygons.  

\smallskip
When the polygon is {\it rational} (in particular if all angles are in $2\pi\Q$ and the polygon is simply connected) or
the holonomy of the flat metric is {\it rational},  then the phase spaces of the billiard flow and of the geodesic flow are foliated by invariant surfaces. 
The existence of a periodic orbit on a dense set of invariant surfaces was first proved by
H.~Masur \cite{Ma86} by a method based on Teichm\"uller theory. 
\subsection{Outline of the argument} 
The argument is by contradiction. This strategy was suggested on the one hand by an interpretation of J. Smillie's proof
in the rational case \cite{Sm99}, \S 4, as a proof by contradiction, and on the other hand  by the results of R.~E. ~Schwartz 
\cite{Sc06}, \cite{Sc_ICM}, \S 5.4, on 
so-called {\it recalcitrant} polygons which reveal the strong instability of  periodic orbits with respect to the parameters 
of the polygon.

We introduce a one-parameter deformation with uniformly bounded sectional curvatures of the solvable (pseudo-homogeneous) metric structure 
of the $3$-dimensional phase space (in fact, of the energy surface) of the geodesic flow on a flat surface with conical singularities on the complement of the cone points. 

From a result of Galperin, Kr\"uger and Troubetzkoy \cite{GKT} we derive that,  under the hypothesis
that there are no regular periodic orbits, the Riemannian manifold with boundary  given by the deformed 
Riemann metric on  unit tangent bundle of the surface (with blow-up at the cone points)  converges to the 
disjoint union of the boundary tori (in the sense that the maximum of the distance function from the boundary 
tori  converges to zero along the deformation).  

We then analyze, under the hypothesis of no regular periodic orbit,
 the skeleton, or cut-locus, of the Riemannian manifold with boundary, defined as the set of point such that their
 distance to the boundary is realized by at least two distinct boundary points. 
 We prove that the cut-locus has
 the homotopy type of a branched surface given by the union of finitely many tori (in number equal to the number of conical points) 
 glued along disks, or, equivalently, by the union of an orientable closed surface of higher genus with finitely many closed 
 topological disks glued along non-intersecting simple closed paths on the surface.
 
 The first Betti number
 of the branched surface, hence of the cut-locus,  diverges along the deformation.
 However, it follows from a result of  A.~G.~Vainshtein,  V.~A.~Efremovich,  and E.~A.~Loginov \cite{VEL78} that the skeleton 
 (which is closed 
 under our hypothesis) is a deformation retract of the manifold, hence its first homology group has a bounded 
 rank. This contradiction proves the theorem.

\section*{Acknowledgements} I am extremely grateful to Rich Schwartz and Pascal Hubert for several questions and comments on a first draft of the paper and
to Rich, David Fisher, Carlos Matheus, Bram Petri, and Anton Zorich for several conversation which helped me to improve the topological part
of the argument. I owe to Pascal Hubert the observation that the argument could perhaps prove existence of infinitely many periodic orbits.

\section{Flat surfaces}

A flat cone surface is a closed surface that is locally Euclidean at all but finitely many singular points. Near each of these singular points
the surface is locally isometric to a {\it conical plane}, that is to the metric of a planar cone with the singular point corresponding to the vertex of
the cone.

More precisely, let $(M,r)$ be a closed orientable surface (of genus $g\geq 0$ and 
Euler characteristic $\chi=2-2g$) endowed with a flat metric $r$ of finite
total area with a finite number of conical singularities \cite{Tr86}, \cite{Thu98}, 
\cite{Ve93}. 

A point $p\in M$ is called a {\it conical point} of 
parameter $\alpha>-1$ if there exists a coordinate $z:U\to\C$, defined 
on a neighborhood $U\subset M$ of $p$, such that $z(p)=0$ and the metric 
(as a a quadratic form) can be written in coordinates as 
$$
r^2(z) = |z^{\alpha}dz|^2\,\,, \quad z\in U\,.
$$
If $p\in M$ is a conical point of parameter $\alpha_p>-1$, the total angle at $p$ is equal to $\theta_p:=2\pi (\alpha_p +1)$ and the {\it concentrated curvature} at $p$ is $K_p := -2 \pi \alpha_p \delta_p$. With this definition the Gauss-Bonnet formula holds. 

A conical point $p\in M$ will be called a conical singularity if $\alpha_p\not=0$. Let $C_r\subset M$ be the subset of all conical singularities of the 
metric $r$ on $M$.
We have
\begin{equation}
\label{eq:GB} 
\frac{1}{2\pi} \sum_{p\in C_r} K_p= \frac{1}{2\pi} \int_M K  = 
\chi \,(=2-2g) \,\,.
\end{equation}
Let $S_r(M)$ be the unit tangent bundle of singular Riemannian manifold $(M,r)$.
Let $\hat M_r:= M\setminus C_r$ and let $\hat S_r(M):=S\left(\hat M_r,r\right)$ be the 
unit tangent bundle for the open Riemannian surface $\left(\hat M_r, r\right)$, that is,
$\hat S_r(M)= S_r(M) \vert \hat M_r$, the restriction of the unit tangent bundle $S_r(M)$ to the complement of the set of conical points. 
By definition 
$$
S_r(M) =  \hat S_r(M) \cup  (C_r \times \T) \,.
$$

Since $(\hat M_r,r)$ is flat, there is natural holonomy homomorphism, 
$$
H_r :\pi_1\left(\hat M_r, \ast\right) \to SO(2,\R) \,,
$$
($\ast$ denotes an arbitrary fixed regular point) given by the parallel transport along closed paths with respect to the Riemannian connection of the metric. Let $p\in C_r$ be a conical singularity 
and let $a_p\in \pi_1\left(\hat M_r,\ast\right)$ be the equivalence class represented by a closed path in $\hat M_r$ with index $1$ with respect to $p$ and index $0$ with respect to all $p'\in C_r$, $p'\not=p$. We have
$$
H_r(a_p) =  \begin{pmatrix} \cos\theta_p & -\sin\theta_p\\
\sin\theta_p &  \cos\theta_p\end{pmatrix}\,\,.
$$
Thus, if the holonomy homomorphism is trivial then all angles at conical singularities are integer multiples of $2\pi$. In this case $g\geq 2$ and
the metric $r$ is induced by a holomorphic (abelian) differential on $M$
vanishing at $C_r$. In addition, if the square of the holonomy homomorphism 
$H_r^2: \pi_1\left(\hat M_r, \ast\right) \to SO(2,\R)$ is trivial then all cone
angles are integer or half integer multiples of $2\pi$. In this case the metric $r$ is induced by a holomorphic quadratic differential on $M$. In all such cases the geodesic flow has a periodic orbit on
a dense set of invariant surfaces \cite{Ma86} and it is uniquely ergodic on almost all invariant surfaces \cite{KMS}. 

\begin{definition} The holonomy homomorphism $H_r :\pi_1\left(\hat M_r, \ast\right) \to SO(2,\R)$ is
called rational if for every $a \in \pi_1\left(\hat M_r, \ast\right)$ the angle of the rotation $H_r(a)$ is a rational multiple of $2\pi$ or, in other terms, the rotation $H_r(a)$ has finite period in $SO(2,\R)$
\end{definition}

\medskip
The unit tangent bundle $\hat S_r(M)$ is an open {\it pseudo-homogeneous} $3$-manifold
(in the sense of \cite{Fo02b}) modeled on the $3$-dimensional solvable Lie algebra $\mathfrak s$ with generators $\{X, Y, \Theta\}$ satisfying the commutation relations
\begin{equation}
\label{eq:commone}
[X,Y]=0\,, \quad [\Theta,X]= Y \,, \quad [\Theta,Y]=-X\,.
\end{equation}
In fact, there is a natural action of the group $SO(2,\R)$ which preserves the
fibers of the bundle $\hat S_r(M) \to \hat M_r$ and let $\Theta$ be the generator of this action. Let $R_{\theta}\in SO(2,\R)$ be the rotation of angle $\theta\in \R$ in the positive direction. Let $X$ be the generator of the geodesic flow 
$\{G_t\}$ and let $Y$ be the generator of the flow $\{G^{\perp}_t\}$ defined as follows:
$$ 
G^{\perp}_t := \{R^{-1}_{\frac{\pi}{2}} \circ G_t \circ 
R_{\frac{\pi}{2}}\}\,, \quad \text{ for all }t\in\R \,.
$$ 
It is a standard fact that, since $\left(\hat M_r,r\right)$ is flat, the commutation relation \pref{eq:commone} hold for the frame $\{X,Y,\Theta\}$ on $\hat S_r(M)$ \cite{ST}, 
\S 7.2, \cite{GK80} or \cite{Fl92}. Although $\hat S_r(M)$ is pseudo-homogeneous, the geodesic vector field $X$ is not a renormalizable element of the Lie algebra $\mathfrak s$ (in the sense of \cite{Fo02b}). Hence the renormalization scheme of \cite{Fo02b} cannot be applied as such to the study of the dynamics of the geodesic flow. In fact, the Lie algebra $\mathfrak s$ has no renormalizable elements in the sense of \cite{Fo02b}. We recall that the notion of a renormalizable element given there is a restricted one: it requires the
existence of a volume preserving automorphism of the Lie algebra with a one-dimensional strong unstable (or strong stable) space. For examples,  all elements of Abelian Lie
algebras are renormalizable, unipotent elements of the Lie algebra $\mathfrak sl(2, \R)$ are (they are often called {\it horospherical}), as well as non-central element of the 
Heisenberg Lie algebra. 

\section{ The height function}

\begin{definition}
\label{def:openstrip}
An \emph{open strip} of width $\epsilon >0$ and height $h>0$ based at $v\in S_r(M)$ is an embedding $E_v : (-\epsilon/2, \epsilon/2) \times  (0, h) \to  \hat S_r(M)$ defined as follows:
\begin{equation}
\label{eq:embed}
E_v (t,s)  = G_t G^\perp_s  v   \,, \quad (s,t) \in (-\epsilon/2, \epsilon/2) 
\times  (0, h) \,.
\end{equation}
\end{definition}
By definition if $E_v$ is an open strip the projection of its range on $\hat M_r$ is an open
flat rectangle which does not contain cone points in its interior.
\begin{definition}
\label{def:height}
The \emph{height function} $h: \R^+ \times S_r(M) \to \R^+ \cup \{+\infty\}$ is  defined as follows. 
For any $(\epsilon, v) \in  \R^+  \times S_r(M)$, the height $h(\epsilon, v)$ is  the maximal height
of an open strip of width $\epsilon>0$ based at $v\in S_r(M)$.
\end{definition}

\begin{lemma} 
\label{lemma:heightfunction}
The height function is everywhere finite and upper semicontinuous.
\end{lemma}
\begin{proof} It was proved in \cite{GKT} (Corollary 2 and Theorem 4) that 
the (forward) orbit of any $v \in S_r(M)$ is either periodic or else its projection onto
the surface $M$ contains cone points in its closure. This result implies that the
height function is everywhere finite. We recall their argument below
for the convenience of the reader.

\smallskip
By a continuity argument,  for every $h\in \R^+$, the set 
\begin{equation}
\label{eq:semicont}
\{ (\epsilon,v) \in \R^+ \times S_r(M) \vert  h(\epsilon, v) < h \} \,\, \text{ \rm is open in }  \,\,
 \R^+\times S_r(M)\,,
\end{equation}
hence the height function is upper semicontinuous. We will proved below by contradiction
that the height function is everywhere finite.

By the semicontinuity property,  for every $(\epsilon, h)\in \R^+\times \R^+$, the set 
$$
\{v \in S_r(M) \vert  h(\epsilon, v) <h\} \,\, \text{ \rm  is open in } S_r(M)\,.
$$
It follows that the set $K_\epsilon :=\{v \in S_r(M) \vert  h(\epsilon, v) =+ \infty\}$ is closed 
since it is a countable intersection of closed sets.  Let us assume that there exists 
$\epsilon \in \R^+$ such that $K_\epsilon \not = \emptyset$. The set $K_\epsilon$ is compact, 
invariant under the geodesic flow and the the restriction of the forward geodesic flow to 
$K_\epsilon$ is continuous, since the projection on to $M$ of the forward trajectory of any vector in $K_\epsilon$ stays uniformly away from the cone points.  
Since in a minimal system any point is uniformly recurrent and since any compact topological system has a minimal subsystem, any continuous transformation or semi-flow $\{\Phi_t\}$ of a compact space $X$
has a \emph{uniformly recurrent }point in the following sense (see for instance H.~Furstenberg
\cite{F}). A point $x\in X$ is
uniformly recurrent if for any neighborhood $\Cal U_x \subset X$ there exists a constant
$C>0$ and a diverging sequence $\{t_n\} \subset \R^+$ such that
\begin{equation}
 0\leq t_{n+1} -t_n \leq C \quad \text{ \rm and } \quad  \Phi_{t_n} (x) \in \Cal U_x \,, 
\quad \text{ \rm for all } n\in \N\,.
\end{equation}
It follows that there exists a  vector $v\in K_\epsilon$ uniformly recurrent under the forward
geodesic flow. Let $\epsilon_v>0$ be defined as the maximum $\epsilon >0$ such that 
$h(\epsilon, v) =+\infty$ (which exists since the set  $\{ \epsilon \in \R^+ \vert  h(\epsilon, v) 
=+\infty\}$ is closed) and let
$$
\Sigma_v :=  E_v \left(  (-\epsilon_v/2, \epsilon_v/2)  \times \R^+ \right) \subset  S_r(M)\,.
$$
Let $\Cal E^l_v$ and $\Cal E^r_v\subset M$ be the unbounded connected 
components (left and right) of  the boundary $\partial \Sigma_v\subset S_r(M)$. 
Let $S_v\subset \hat M_r$ be the projection of the strip $\Sigma_v$ onto the (open)
flat surface $\hat M_r$. Since $v\in S_r(M)$ is recurrent but non-periodic, the set  
$\partial S_v \cap S_v$ contains the boundaries of parallelograms with two arbitrarily 
long parallel edges. It follows that for every $L>0$ there exist $w^l_L \in \Cal E^l_v$ 
and $w^r_L \in \Cal E^r_v$ such that 
\begin{equation}
\label{eq:longedges}
h(\frac{\epsilon_v}{2}, w^l_L)\geq L \quad \text{\rm and  } \quad  
h(\frac{\epsilon_v}{2}, w^r_L)\geq L\,.
\end{equation}
However, since $v\in S_r(M)$ is uniformly recurrent, the above property cannot hold.
In fact, by definition $h(\epsilon,v) <+\infty$ for any $\epsilon >\epsilon_v$, hence for
$\epsilon = 3\epsilon_v/2$ it follows that 
$$
h^+_v := h(\frac{\epsilon_v}{2}, G^\perp_{\frac{\epsilon_v}{2}}(v)) <+\infty 
\quad \text{ \rm or }  \quad 
h^-_v := h(\frac{\epsilon_v}{2}, G^\perp_{-\frac{\epsilon_v}{2}}(v)) <+\infty \,.
$$
By continuity,  there exists a neighborhood $\Cal U_v \subset S_r(M)$ such that 
for all $w\in \Cal U_v$, 
$$
h(\frac{\epsilon_v}{2}, G^\perp_{\frac{\epsilon_v}{2}}(w)) < 2h^+_v   \quad \text{ \rm or }
 \quad h(\frac{\epsilon_v}{2}, G^\perp_{-\frac{\epsilon_v}{2}}(w)) <2 h^-_v  \,.
$$
By the uniform recurrence property, there exist $C>0$ and a diverging sequence 
$\{t_n\}$ such that $t_{n+1}-t_n \leq C$ and $ v_n:=G_{t_n} (v) \in \Cal U_v$, hence 
$$
h(\frac{\epsilon_v}{2}, G^\perp_{\frac{\epsilon_v}{2}} G_{t_n}(v)) < 2 h^+_v \,,
\quad \text{ \rm for all } n\in \N \,, 
$$
or  
$$
 \quad h(\frac{\epsilon_v}{2}, G^\perp_{-\frac{\epsilon_v}{2}}G_{t_n}(v)) < 2 h^-_v  \,,
 \quad \text{ \rm for all } n\in \N \,.
$$
Since $t_{n+1}-t_n \leq C$, it follows that 
$$
h(\frac{\epsilon_v}{2}, w) \leq   C + 2 h^+_v\,,      \quad \text{ \rm for all } w \in \Cal E^l_v\,,
$$
or 
$$
h(\frac{\epsilon_v}{2}, w) \leq    C+ 2h^-_v \,,     \quad \text{ \rm for all } w \in \Cal E^r_v\,,
$$
in contradiction with the conclusion \pref{eq:longedges}. Such a contradiction 
completes the argument.

\end{proof}

\begin{corollary} \label{cor:maxheigth}
For any $\epsilon >0$, the following holds:
$$
H(\epsilon) := \sup_{v\in S_r(M)} h(\epsilon, v) =\max  _{v\in S_r(M)} h(\epsilon, v) < +\infty\,.
$$
\end{corollary}

\begin{proof}
 It follows from the upper semicontinuity property of the height function.  Let $(h_n)$ be a
 non-decreasing sequence of real numbers converging to $H(\epsilon)\geq 0$. 
 By semicontinuity  for every  $n\in \N$, the set
 $$
 K_n := \{ v\in  S_r(M) \vert h(\epsilon, v ) \geq h_n\}
 $$
 is closed in $S_r(M)$, hence compact. By construction $K_n\not =\emptyset$
 and  $K_{n+1} \subset K_n$ for all $n\in \N$. Since the sets $K_n$ are
 compact, there exists $v\in S_r(M)$ such that $v\in K_n$ for all $n\in \N$, 
 hence $h(\epsilon, v)=H(\epsilon)$. Thus  $H(\epsilon)$ is equal to the maximum
  of the height function $h(\epsilon, \cdot)$ on $S_r(M)$. Since  the height function 
 is everywhere finite by Lemma \ref{lemma:heightfunction},   it follows in particular 
 that $H(\epsilon)$ is finite. 
 \end{proof}
 
\section{A bounded curvature deformation} We consider a constant volume deformation $\{X_s, Y_s, \Theta_s\}$
defined by the following rescaling  of the frame $\{X, Y, \Theta\}$: for all $s>0$, let
\begin{equation}
\label{eq:deformation}
X_s := e^s  X  \,, \quad Y_s := Y \,, \quad \Theta_s := e^{-s} \Theta \,.
\end{equation}
Let $R_s$ be the Riemannian metric on the open $3$-manifold $\hat S_r(M)$ defined by the condition 
that the frame $\{X_s, Y_s, \Theta_s\}$ is orthonormal, that is, the unique metric such that
\begin{equation}
\label{eq:Rtmetric}
\begin{aligned}
&R_s(X_s,X_s)=R_s(Y_s,Y_s)=R_s(\Theta_s,\Theta_s)=1 \,, \\
&R_s(X_s,Y_s)=R_s(X_s,\Theta_s)=R_s(Y_s,\Theta_s)=0\,.
\end{aligned}
\end{equation}
A straightforward calculation yields the following result.

\begin{lemma}  
\label{lemma:bc}
For each $s>0$, the metric $R_s$ is pseudo-homogeneous modeled on the
solvable Lie algebra $\mathfrak s_s$ with generators $\{X_s, Y_s, \Theta_s\}$ satisfying the commutation relations
$$
[X_s, Y_s]=0\,, \quad [\Theta_s,X_s]= Y_s \,, \quad [\Theta_s,Y_s]=- e^{-2s} X_s\,.
$$
The  structural constants of the basis $\{X_s, Y_s, \Theta_s\}$ converge to the structural constant of the Heisenberg Lie algebra. In particular, all sectional curvatures of the metric $R_s$ are uniformly 
bounded, with all derivatives, on $\hat S_r(M)$ for all $s>0$.
\end{lemma}
\begin{proof} By the commutation relation the structure constant of the basis $\{X_s, Y_s, \Theta_s\}$  are given by matrices
$$
C^1 = \begin{pmatrix} 0 & 0 & 0 \\ 0 & 0 & e^{-2s} \\ 0 & -e^{-2s} & 0 \end{pmatrix} \,, \quad C^2 = \begin{pmatrix} 0 & 0 & -1 \\ 0 & 0 & 0 \\ -1 & 0 & 0 \end{pmatrix} \,,
 \quad C^3 = \begin{pmatrix} 0 & 0 & 0 \\ 0 & 0 & 0 \\ 0 & 0 & 0 \end{pmatrix}\,,
$$
 hence they are uniformly bounded for all $s>0$.  Since the Levi-Civita connection, the Riemann tensor and the sectional
 curvature are given by linear or quadratic expressions in the structure constants \cite{Mil}, they are uniformly bounded as well
 with uniformly bounded derivatives of all orders.
\end{proof} 

 For all $s>0$, let $d_s (v,\partial \hat S_r(M))$ denote the distance with respect to the deformed 
 Riemannian metric $R_s$ on $\hat S_r(M)$ of a vector $v\in \hat S_r(M)$ from the boundary
 $\partial \hat S_r(M) \equiv  C_r \times \T$ of the restriction of the unit tangent bundle to
 the open surface $\hat M_r$. Let  $d_s (\hat S_r(M), \partial \hat S_r(M))$ be the $R_s$-diameter of the
 $3$-manifolds relative to $\partial \hat S_r(M)$:
 $$
 d_s (\hat S_r(M), \partial \hat S_r(M)) = \sup_{v\in S_r(M)} d_s (v,\partial \hat S_r(M))\,.
 $$

\begin{lemma} 
\label{lemma:zerodiam}
If the geodesic flow $\{ G_t\}$ has no periodic orbits, the limit of the
relative diameter along the deformation $\{R_s\}$ is zero, that is, 
\begin{equation}
\label{eq:zerodiam}
\lim_{t\to +\infty} d_s (\hat S_r(M), \partial \hat S_r(M))  = 0 \,.
\end{equation}
\end{lemma}
\begin{proof} For any $\epsilon>0$, let $\Sigma_v(\epsilon)$ be the  open strip of width 
$\epsilon>0$ of maximal height based at $v\in \hat S_r(M)$, that is,
$$
\Sigma_v(\epsilon):=   E_v \left(  (-\epsilon_v/2, \epsilon_v/2)  \times (0, h(\epsilon,v)\right) 
\subset \hat S_r(M)\,.
$$
Let $S_v(\epsilon) \subset \hat M_r$ the projection of $\Sigma_v(\epsilon)$ onto the open surface 
$\hat M_r$. Since $\Sigma_v(\epsilon)$ has maximal height and the geodesic flow has no periodic orbits, there is a cone point $p\in C_r$ in the boundary $\partial S_v(\epsilon) \subset M$. 
By definition of the embedding $E_v$ in formula  \pref{eq:embed} and of the deformation 
$\{R_s\}$ in formula \pref{eq:deformation}, parallel transport in $S_r(M) \vert S_v(\epsilon)$ 
yields a path joining $v\in \Sigma_v$ to $\{p\} \times \T \subset \partial \hat S_r(M)$ of 
$R_s$-length  at most  $\sqrt{  \epsilon^2/4 + e^{-2s} h(\epsilon, v)^2}$. By Corollary \ref{cor:maxheigth}, there exists $H(\epsilon)>0$ such that $h(\epsilon, v) \leq H(\epsilon)$ for all $v\in \hat S_r(M)$. It follows 
that for all $\epsilon>0$ and all $s>0$,
$$
d_s \big (\hat S_r(M), \partial \hat S_r(M)\big) \leq \sqrt{ \epsilon^2/4 + e^{-2s} H(\epsilon)}\,.
$$
which implies the desired conclusion \pref{eq:zerodiam}.
\end{proof}

After a blow up of the flat metric on the surface $M$ near the cone point, the boundary $\partial \hat M_r$ is a union of finitely many circles, hence the boundary $\partial \hat S_r(M)$
of  its unit tangent bundle $\hat S_r(M)$ is the union of finitely circles in number equal to the cardinality $\# C_r$ of the set of conical points.  In fact, near a conical point
$p$ of total angle $2\pi (\alpha+1)$, the metric has the form $r = \vert z^\alpha dz \vert$ with respect to a (canonical) holomorphic local coordinate $z$ such that
$z(p)=0$. In polar coordinates with geodesic radial coordinate $\rho = \vert z \vert^{\alpha+1} / (\alpha+1)$  the metric has the form
$$
r^2 = d \rho^2 + [ 2\pi (\alpha+1)]^2 d\phi^2 \,.
$$
Thus after blow-up of the singularity at the vertex the boundary of the open cone is the circle $\rho=0$ of total length $2\pi (\alpha+1)$ and therefore the boundary of the unit
tangent bundle of the open cone is a torus. In coordinates $(\phi, \theta) \in \R^2$, the boundary torus is the flat torus given by the lattice generated by the vectors
$2\pi (1, \alpha+1)$ and $2\pi (0,1)$. In other terms, for all $(\phi, \theta) \in \R^2$ we have 
$$
(\phi, \theta) \equiv (\phi +2 \pi, \theta + 2\pi (\alpha+1)) \quad \text{ and } \quad  (\phi, \theta) \equiv (\phi, \theta + 2\pi) \,.
$$

\begin{lemma}  \label{lemma:convex} For every $s\geq 0$, the boundary $\partial \hat S_r(M)$ with inward normal is locally geodetically concave with respect to the metric $R_s$, in the sense that its second fundamental form 
is negative semi-definite.
\end{lemma} 
\begin{proof} The argument is an elementary calculation that we reproduce here for the convenience of the reader. 

It is enough to prove the result for each connected component of $\partial \hat S_r(M)$, hence for a single cone point. The flat metric at the cone point $p$ with respect to a complex
coordinate $z$ such that $z(p)=0$ is of the form $r= \vert z^\alpha dz \vert$ ($\alpha >-1$). We then compute formulas for the geodesic vector field $X$, orthogonal geodesic vector field $Y$ and the generator $\Theta$ of rotations  in polar coordinates $(z, e^{i\theta} ) \equiv  (\rho, \phi, \theta)$ on the unit tangent bundle of the cone:
$$
\begin{aligned}
X &=  \cos (\theta -(\alpha+1) \phi)  \frac{\partial}{\partial \rho} + \frac{1}{(\alpha+1)\rho}   \sin (\theta-(\alpha+1) \phi) \frac{\partial}{\partial \phi}  \,, \\
Y&=  -\sin (\theta -(\alpha+1) \phi)  \frac{\partial}{\partial \rho} + \frac{1}{(\alpha+1)\rho}    \cos (\theta-(\alpha+1) \phi) \frac{\partial}{\partial \phi}    \Big)\,, \\
\Theta &=  \frac{\partial}{\partial \theta} \,. 
\end{aligned} 
$$
Since by definition of the metric $R_s$, the vector fields $X_s= e^s X$, $Y_s= Y$ and $\Theta_s= e^{-s} \Theta$ are orthonormal and 
$$
\frac{\partial}{\partial \phi} = (\alpha+1)\rho \Big( \sin (\theta-(\alpha+1) \phi)  X +  \cos (\theta-(\alpha+1) \phi)  Y \Big)\,,
$$
it follows that the inner unit normal for the metric $R_s$ 
to the boundary of the cone is given by the formula
$$
N_s =  \frac{ e^{s} \cos (\theta -(\alpha+1) \phi) X_s -  \sin (\theta -(\alpha+1) \phi) Y_s}{\sqrt{e^{2s} \cos^2 (\theta -(\alpha+1) \phi)  + \sin^2 (\theta -(\alpha+1) \phi)   }  }\,. 
$$
Let $\xi = \theta -(\alpha+1) \phi$. We have
$$
\begin{aligned}
\frac{\partial}{\partial \xi}  \Big(\frac{  e^{s} \cos \xi  }{\sqrt{e^{2s} \cos^2\xi  + \sin^2 \xi   } } \Big) =  \frac{  - e^{s} \sin \xi  }{ (e^{2s} \cos^2\xi  + \sin^2 \xi)^{3/2}  }  \,, \\
\frac{\partial}{\partial \xi}  \Big(\frac{ - \sin \xi  }{\sqrt{e^{2s} \cos^2\xi  + \sin^2 \xi   } } \Big) =  \frac{  -e^{2s} \cos \xi  }{ (e^{2s} \cos^2\xi  + \sin^2 \xi)^{3/2}  }  \,, 
\end{aligned} 
$$
The tangent space to the boundary torus is equal to $\R \frac{\partial}{\partial \phi} \oplus \R \Theta$ and we have the formula
$$
 \frac{\partial}{\partial \phi} := \sin (\theta -(\alpha+1) \phi) X_s  +  e^s\cos (\theta -(\alpha+1) \phi) Y_s \,.
$$
By the commutation relations, and in particular since $[X_s, Y_s]=0$, it follows that the range of the map  $dN_s$ is a subspace of $\R X_s \oplus \R Y_s$, 
hence it is contained in $\Theta^\perp$, and 
$$
dN_s (  \frac{\partial}{\partial \phi}  ) = \frac{ (\alpha+1)  [  e^{s} \sin (\theta -(\alpha+1) \phi) X_s + e^{2s} \cos (\theta -(\alpha+1) \phi)   Y_s] }{ (e^{2s} \cos^2 (\theta -(\alpha+1) \phi) + \sin^2  (\theta -(\alpha+1) \phi))^{3/2}   }    
$$
Finally the negative of the second fundamental form has matrix 
$$
\begin{pmatrix} R_s( dN_s (  \frac{\partial}{\partial \phi}) ,  \frac{\partial}{\partial \phi} )  & R_s( dN_s (\frac{\partial}{\partial \phi}), \Theta) \\
 R_s( dN_s ( \Theta) ,  \frac{\partial}{\partial \phi} )  &  R_s( dN_s ( \Theta) ,  \Theta)   \end{pmatrix}  = 
 \begin{pmatrix} R_s( dN_s (  \frac{\partial}{\partial \phi}) ,  \frac{\partial}{\partial \phi} )  & 0 \\
 0 &  0  \end{pmatrix} 
$$
with
$$
R_s\Big( dN_s (  \frac{\partial}{\partial \phi} ),  \frac{\partial}{\partial \phi} \Big) = (\alpha+1) e^s  \frac{ [   \sin^2 (\theta -(\alpha+1) \phi)  + e^{2s} \cos^2 (\theta -(\alpha+1) \phi)   ] }{ (e^{2s} \cos^2 (\theta -(\alpha+1) \phi) + \sin^2  (\theta -(\alpha+1) \phi))^{3/2}   }  \,,
$$
hence the second fundamental form is negative semi-definite and the boundary is infinitesimally geodetically concave.

\end{proof} 

\section{ The cut-locus}

Let us consider the  cut-locus or skeleton $\mathcal C(s) \subset \hat S_r(M)$ for the metric $R_s$ defined as follows: a point $p\in \mathcal C(s)$ if and only if $p\in \hat S_r(M)$ and there exists at least two distinct geodesics $\gamma_1, \gamma_2: [0, 1] \to  \hat S_r(M)$ such that
$$
\begin{aligned} 
\gamma_1 (0)&= \gamma_2(0) = p   \qquad \text{ and }  \qquad \lim_{s\to 1-}  \gamma_1(s),  \lim_{s\to 1-}  \gamma_1(s) \in \partial \hat S_r(M)    \\
&\int_0^1  \Vert  \gamma_1'(s) \Vert_{R_s} ds =  \int_0^1  \Vert  \gamma_2'(s) \Vert_{R_s} ds =  d_s (p, \partial \hat S_r(M)) \,.
\end{aligned} 
$$
Let $\Gamma_s (p, \partial \hat S_r(M))$ denote the set of $R_s$-geodesics $\gamma: [0, 1) \to \hat S_r(M)$  such  that 
$$
\gamma(0)=p \,,    \quad  \lim_{t\to 1-}  \gamma_1(s) \in \partial \hat S_r(M)    \quad \text{ and } \quad        \int_0^1  \Vert  \gamma'(s) \Vert_{R_s} ds =  d_s (p, \partial \hat S_r(M))  \,.
$$
The cut-locus, or skeleton, is by definition the set 
$$
 \mathcal C(s) := \{ p \in \hat S_r(M) \vert    \#   \Gamma_s (p, \partial \hat S_r(M)) >1\}\,.
$$

\begin{lemma} If the geodesic flow $\{ G_t\}$ has no periodic orbits, there exists $s_0>0$ such that for all $s>s_0$  the set $\Gamma_s (p, \partial \hat S_r(M))$ is discrete
and cardinality at most equal to the number $\# C_r$ of cone points.
\end{lemma} 
\begin{proof} Since the family of metric $\{R_s \vert s>0\}$ is compact, there exists $\delta_0>0$ such that the metric balls in the metric $R_s$ of radius at most $\delta_0$
have geodetically strictly convex boundary. If the geodesic flow $\{ G_t\}$ has no periodic orbits, by Lemma \ref{lemma:zerodiam} there exists $s_0>0$ such that 
$\text{\rm diam}_s(S_r, \partial S_t)< \delta_0$, for all $s \geq s_0$. It follows that, for every $p \in \hat S_r(M)$, the closed ball $\bar B_s(p,r_s)$ in the metric $R_s$ 
of center $p$ and radius $r_s:=d_s(p, \partial \hat S_r(M))$ is geodetically strictly convex. Since by Lemma \ref{lemma:convex} the connected components of the boundary are 
geodetically convex, 
the ball $\bar B_s(p,r_s)$ intersects each connected component of $\partial \hat S_r(M)$ in at most one point. 
\end{proof}

We investigate below the topology of the cut-locus and of its complement.  

\begin{lemma} 
\label{lemma:cutlocus}
If the geodesic flow $\{ G_t\}$ has no periodic orbits,  there exists $s_0>0$ such that for all $s\geq s_0$ the cut-locus  $\mathcal C(s)$ is a closed subset of $\hat S_r(M)$ and its complement $\hat S_r(M)\setminus \mathcal C(s)$ has $\# C_r$ connected components, each diffeomorphic to a product $I \times \T^2$ of an open interval $I \subset \R$ with a $2$-dimensional torus. 
\end{lemma} 
\begin{proof}  There exists $s_0>0$ such that, for all $s \geq s_0$, if $p\not \in \mathcal C(s)$, there exists a unique minimizing geodesic $\gamma_p$  joining $p$ to $\partial \hat S_r(M)$, whose length equals $d_s(p, \partial \hat S_r(M))$. In fact, each component 
of the boundary  $ \partial \hat S_r(M) $ is geodesically convex in $(\hat S_r(M) , R_s)$ and, since  in addition the one-parameter family
of metrics $\{R_s\vert s>0\}$ belongs to a compact set, there exists $\delta_0>0$ such that when $d_s(\hat S_r(M), \partial \hat S_r(M))
< \delta_0$, then there are no focal points of $\partial \hat S_r(M)$ in $\hat S_r(M) \setminus \mathcal C(s)$.  In fact, closed balls of radius at most $\delta_0>0$ are
geodetically strictly convex and the exponential map from their center is injective. 
It follows then from Lemma
\ref{lemma:zerodiam} that there exists $s_0>0$ such that $d_s(\hat S_r(M), \partial \hat S_r(M)) < \delta_0$ for all $s \geq s_0$. 

\smallskip
\noindent If the cut-locus $\mathcal C(s)$ is not closed, since the manifold is compact, there exists a converging sequence $(p_n) \subset \mathcal C(s)$ such that
its limit $p\not\in \mathcal C(s)$. For each $n\in \N$, let $q^{(1)}_n \not = q^{(2)}_n \in \partial \hat S_r(M) $ be closest points to $p_n$.  By compactness it is not restrictive
to assume that $(q^{(1)}_n)$ and $(q^{(2)}_n)$ are convergent to points $q^{(1)}$ and $q^{(2)} \in \partial \hat S_r(M) $. Such points realize the distance of $p$ from 
$\partial \hat S_r(M)$. In fact,  if $d(p, \partial \hat S_r(M)) < d(p,q^{(1)})=d(p,q^{(2)})$, then, for $n\in \N$ large enough, we have that $d(p_n, \partial \hat S_r(M)) <d(p_n,q^{(1)}_n)=
d(p_n,q^{(2)}_n)$ in contradiction with the choice of $q^{(1)}_n$ and $q^{(2)}_n$.  

If $p\not \in \mathcal C(s)$, then $q^{(1)}=q^{(2)}$, hence $q^{(1)}_n$ and $q^{(2)}_n$
belong to the same connected component of the boundary.  However, since connected components of the boundary are geodetically convex, closed balls are geodetically
strictly convex implies $q^{(1)}_n = q^{(2)}_n$, a contradiction that concludes the proof that $\mathcal C(s)$ is closed.

\smallskip
\noindent Under the assumption that there are no focal points, it follows that all other geodesics, different from $\gamma_p$, joining $p$ to $\partial \hat S_r(M)$ have length strictly larger than $d_s(p, \partial \hat S_r(M))$.  In addition, the point $p$ belongs to a
geodesic tubular neighborhood  $\mathcal N$ of a connected component of $\partial \hat S_r(M)$, which is homeomorphic to a torus. There exists therefore a neighborhood
$\mathcal U_p$  of $p$ in $\hat S_r(M)$ such that for all $q \in \mathcal U_p$ there exists a unique minimizing geodesics $\gamma_q \subset \mathcal N$ whose length equals
$d_s(q, \partial \hat S_r(M))$. Finally, since every $p\in  \hat S_r(M)$ is at finite distance from $\partial \hat S_r(M)$,  the complement of the cut-locus is the union of tubular neighborhoods
of the boundary tori, hence it is a union of $\# C_r$ components, each homeomorphic to a product $I \times \T^2$, with $I \subset \R$ an open bounded interval.
\end{proof} 

Let $c:= \# C_r \geq 3$ and let $\mathcal N_1(s), \dots,  \mathcal N_c(s)$ denote the connected components of $\hat S_r(M)\setminus \mathcal C(s)$. The metric
completions $\overline{\mathcal N_1}(s), \dots, \overline{\mathcal N_c}(s)$ of the tubular neighborhoods $\mathcal N_1(s), \dots,\mathcal N_c(s)$ are homeomorphic to 
a product $\overline{I} \times \T^2$, with $\overline{I}$ the closure of the open interval $I\subset \R$.  
It follows that there exists a continuous surjective map
$$
\pi_s : \T^2_1(s) \sqcup \dots  \sqcup \T^2_c(s)    \to   \mathcal C(s) 
$$
from a disjoint union of a family of topological tori  $\{ \T^2_i(s) \vert i\in \{1, \dots,c\} \}$ such that $ \T^2_i(s) \subset \partial \overline{\mathcal N_i}(s)$, for all $i\in \{1, \dots,c\}$.
In fact, for each $i\in \{1, \dots,c\}$  the inclusion $J_i : \mathcal N_i(s) \to \hat S_r(M)$ extends by continuity to maps $\bar J_i : \overline{ \mathcal N_i} (s) \to S_r$ and the boundary 
$\partial  \overline{ \mathcal N_i} (s)$ is a union of two connected components $\partial^{(-)}  \overline{ \mathcal N_i} (s) \cup \partial^{(+)}  \overline{ \mathcal N_i} (s)$,
each homeomorphic to a torus, such that 
$$
\bar J_i \left( \partial^{(-)}  \overline{ \mathcal N_i} (s)  \right)  \subset \partial \hat S_r(M) \quad \text{ and } \quad \bar J_i \left( \partial^{(+)} 
 \overline{ \mathcal N_i} (s)  \right)  \subset  \mathcal C(s) \,.
$$
We then define $ \T^2_i(s) :=   \partial^{(+)}  \overline{ \mathcal N_i} (s) $ and $\pi_s \vert  \T^2_i(s) = \bar J_i$, for all $i\in \{1, \dots,c\}$.

\medskip
\noindent Let $\mathcal{C}^{\ast}(s) \subset \mathcal C(s)$ be the set of points of $\hat S_r(M)$ which are equidistant from exactly $2$
(possibly equal) boundary tori, that is, 
$$
 \mathcal C^\ast (s) := \{ p \in \hat S_r(M)  \vert   \#  \Gamma_s (p, \partial \hat S_r(M)) = 2 \}\,.
$$
\begin{lemma} 
\label{lemma:smooth}
If the geodesic flow $\{ G_t\}$ has no periodic orbits, there exists $s_0:= s_0(M)>0$ such that,  for all $s>s_0$, every connected component  of 
the set $ \mathcal C^\ast (s) \subset \mathcal C(s)$ of regular points is open in the cut-locus $\mathcal C(s)$, 
 and it is a smooth $2$-manifold with piece-wise smooth boundary. 
\end{lemma} 
\begin{proof}  Let $p\in \mathcal C^{\ast}(s)$ and let $C_1:=C_1(p)$ and $C_2:=C_2(p)$ denote the connected components of  $\partial \hat S_r(M)$ such that there exist $p_1\in C_1$
and $p_2 \in C_2$ such that 
$$
d_s (p, p_1) = d_s(p, p_2) =   d_s(p, \partial \hat S_r(M)) \,.
$$
There exists a neighborhood $\mathcal U_p \subset \hat S_r(M)$ of $p$ such that 
the functions $d_s (\cdot, C_1)$ and $d_s(\cdot, C_2)$ are smooth on $\mathcal U_p$. 
It follows from the hypothesis that $\# \Gamma_s(p, \partial \hat S_r(M)) =2$ that  there exists a neighborhood 
$\mathcal V_p\subset \mathcal U_p$ such that 
$$
\mathcal C(s) \cap  \mathcal V_p = \{ p' \in \mathcal V_p \vert  d_s (p', C_1)= d_s (p', C_2) \} \,,
$$
hence we conclude that  $\#  \Gamma_s (p', \partial \hat S_r(M)) = 2$ for all $p' \in \mathcal C(s) \cap  \mathcal V_p$. Since
the above argument holds for any point $p \in \mathcal C^\ast (s)$ we have proved that $ \mathcal C^\ast (s)$ is
open in $ \mathcal C(s)$.

\noindent Let $d_s:= d_s (\hat S_r(M), \partial \hat S_r(M))$ denote the maximal distance of points of $\hat S_r(M)$ from the boundary $ \partial \hat S_r(M)$, that is, 
the relative diameter of the pair $(\hat S_r(M), \partial \hat S_r(M))$, with respect to the Riemannian metric $R_s$ on $\hat S_r(M)$.  Since the metric $R_s$ have uniformly 
bounded sectional curvatures, there exists $d_0$ such that all geodesic balls of the metric $R_s$ contained in $\hat S_r(M)$ of radius less than $d_0$ are strictly convex.  

\noindent  By Lemma~\ref{lemma:zerodiam}  there exists $s_0$ such that $d_s < d_0$  for all $s>s_0$. 
If $p\in \mathcal C^\ast (s)$, then there exists a (strictly convex) geodesic ball $B_s(p, d) \subset \hat S_r(M)$ of radius $d < d_s$ such that 
$\partial B_s(p, d) \cap \partial \hat S_r(M)= \{p_1, p_2\}$ with $p_1 \not = p_2$. 
Since $C_1 =C_1(p)$ and $C_2:=C_2(p)$ are convex and the geodesic balls of radius $d<d_s$ are strictly convex, the distance functions $d_s(\cdot, C_1)$ and $d_s(\cdot, C_2)$  from
$C_1$ and $C_2$ respectively have the same gradient at $p$ as the distance functions $d_s(\cdot, p_1)$ and $d_s(\cdot, p_2)$ from the points $p_1$ and $p_2$. Since $p_1\not =p_2$
the distance functions $d_s(\cdot, C_1)$ and $d_s(\cdot, C_2)$, from $C_1$ and $C_2$ respectively, have linearly independent gradients at $p$, hence by the implicit function 
theorem the set $\{x \in \hat S_r(M) \vert d_s(x, C_1) =d_s(x,C_2)\}$ is
a local manifold at $p$.  It follows that every connected component of the set $\mathcal C^\ast(s)$ is a smooth surface. 

The singular set $\mathcal C(s) \setminus \mathcal C^\ast(s)$ is stratified according to the cardinality of the set of point on the boundary which realize the distance to the boundary.
The highest dimensional stratum is given by the subset of points $p \in \hat S_r(M)$ such that there exist three (not necessarily distinct) connected components of the boundary 
$C_1$, $C_2$ and $C_3$ and three distinct points $p_1\in C_1$, $p_2\in C_2$ and $p_3 \in C_3$ such that  
$$
d_s (p, p_1) = d_s(p, p_2) = d_s(p,p_3) = d_s(p, \partial \hat S_r(M)) \,.
$$
Since $p_1$, $p_2$ and $p_3$ are distinct and belong to the boundary of a strictly convex (nearly Euclidean) ball, the differentials at $p$  of the distance functions $d_s(\cdot,C_1)$, 
$d_s(\cdot, C_2)$ and $d_s(\cdot, C_3)$ are equal to the differentials of the distance functions from $d_s(\cdot, p_1)$, $d_s(\cdot, p_2)$  and $d_s(\cdot, p_3)$ from the points $p_1$, $p_2$ and $p_3$
respectively, hence the differential of $d_s(\cdot,C_1)-d_s(\cdot, C_2)$ and $d_s(\cdot, C_2)-d_s(\cdot, C_3)$ are linearly independent. We conclude that the set
$$
\{ x \in \hat S_r(M) \vert  d_s(x,C_1)-d_s(x,C_2)=d_s(x,C_2)-d_s(x,C_3)=0\}
$$
is locally a $1$-dimensional manifold.  Its closure is a compact manifold since the space $\hat S_r(M)$ is compact.  
The argument is therefore complete. 
\end{proof} 

\begin{lemma} 
\label{lemma:homotopy_trivial}
If the geodesic flow $\{ G_t\}$ has no periodic orbits, there exists $s_0:= s_0(M)>0$ such that,  for all $s>s_0$, every connected component  of 
the set $ \mathcal C^\ast (s)$ is homotopically trivial in the cut-locus $\mathcal C(s)$, hence it is homeomorphic to a closed disk with finitely many holes.
\end{lemma} 
\begin{proof}
\noindent We will prove that the closure of each connected component of ${\mathcal C}(s)^\ast $ is a subset of a closed contractible set. Since it is
homeomorphic to a closed subsets with piecewise differentiable boundary  of a boundary component of $\hat S_r(M)$, a $2$-torus, it will follow that 
it is homeomorphic to a closed disk with (possibly) finitely many holes.
 
 \noindent Let $p\in \mathcal C^{\ast}(s)$ and let $C_1:=C_1(p)$ and $C_2:=C_2(p)$ denote the connected components of  $\partial \hat S_r(M)$ such that 
 there exist $p_1\in C_1$ and $p_2 \in C_2$ such that 
$$
d_s (p, p_1) = d_s(p, p_2) =   d_s(p, \partial \hat S_r(M)) \,.
$$
 Let $\gamma$ be the unique length minimizing geodesic segment (in the ball $B_s(p,r)$)  with endpoints $p_1$ and $p_2$. The set of points equidistant from $p_1$ and $p_2$ 
within the ball $B_s(p, d)$ is a smooth hypersurface $\mathcal S$ which intersects $\gamma$ in a single point ${\mathcal R} (p)$, thereby defining a continuous map ${\mathcal R}: {\mathcal C}(s)^\ast \to \Gamma_s(d)$  from  $ \mathcal C^\ast (s)$  to the union  of the set $ \Gamma_s(d)$ of interior geodesics with endpoints in $\partial \hat S_r(M)$ of $R_s$-length at most~$2d$.  

\noindent The map ${\mathcal R}$ is injective since on the surface of points equidistant from a geodesic (contained in a ball of radius $d$) the normals to  the boundary of $\hat S_r(M)$ at its endpoints have at most one intersection. The image of ${\mathcal R}$ contains at most one point for each geodesic (the point equidistant from the two boundary components), hence the image of the closure of each connected component of $ \mathcal C^\ast (s)$ is homeomorphic to a closed subset of a connected component of the set $\Gamma_s(d)$.

\noindent We then prove that the closure of  each connected component of the set $ \Gamma_s(d)$ is contractible.  In fact, the projection on $M$ of all geodesics in
a given connected component of $ \Gamma_s(d)$ are homotopic to a union $\sigma$ of straight segments connecting cone points
of total $R_s$-length at most $d$ forming angles all at least equal to $\pi$. The latter condition implies that the total angle variation along the piece-wise linear curve $\sigma$ is at most $d e^{-s}$.  Since all geodesics in $ \Gamma_s(d)$ have $R_s$-length at most $d$, it follows that they all have  unit tangent vector component in a neighborhood of size at most $d e^{-s}$ of the union of the unit tangent vectors of $\sigma$, hence their projection on $M$ is contained in the intersection of conical neighborhood of $\sigma$ of angle at most $d e^{-s}$ based at its endpoints. Thus, under the condition that $d e^{-s} <1$, all geodesics in a connected component of $ \Gamma_s(d)$ belong to a contractible subset of $\hat S_r(M)$, hence
they form a contractible set.
\end{proof}

\noindent In fact, we can prove that the closure of every connected component of ${\mathcal C}^*(s)$ is homomorphic to a disk, although it is not necessary for the proof of the
main theorem on existence of periodic orbits.

\begin{addendum} 
Under the hypotheses of Lemma~\ref{lemma:homotopy_trivial} the closure of each connected component of $\mathcal C^*(s)$ is homeomorphic to a closed disk.
\end{addendum}
\begin{proof} Since the closure of every connected component of $\mathcal C^*(s)$ is homeomorphic to a closed disk with finitely many holes, it remains to prove that its boundary 
has a single connected component. Let $C_1$, $C_2$ denote connected components of the boundary of $\hat S_r(M)$ (not necessarily distinct). Let $d_s(\cdot, C_1)$, $d_s(\cdot, C_2)$  be distinct branches of the distance functions to the boundary components $C_1$ and $C_2$ respectively. Let $D_{12}$  denote the connected component of $\mathcal C^*(s)$ such that $ d_s(\cdot, C_1) - d_s(\cdot, C_2)=0$ on $D_{12}$.  Let $\Sigma_{12}$ be the open smooth surface defined by the condition that the distance  $d_s(\cdot, C_1)$ from $C_1$ 
and $d_s(\cdot, C_2)$ from $C_2$ are equal. Clearly by definition $D_{12} \subset  \Sigma_{12}$. If $D_{12}$ is not contractible, then $\Sigma_{12} \setminus D_{12}$ has
a relatively compact  open connected component $D'_{12}$  such that $\partial D'_{12} \subset \partial D_{12} \subset \Sigma_{12}$. By the definition of cut-locus there exists a finite set 
$\mathcal C$ of boundary components of $\hat S_r(M)$ (not necessarily distinct or distinct from $C_1$ and $C_2$)  such that the distinct branches of the distance functions 
$d_s(\cdot, C)$ from $C\in \mathcal C$ have the following property:
$$
D'_{12} \cap  \{ x\in \Sigma_{12} \vert d_s(x, C) < d_s(x, C_1) =d_s(x,C_2) \}  \not =\emptyset\,, \quad \text{ for all }  C \in \mathcal C\,.
$$
However, let $p \in \bar D'_{12}$ be a point which maximizes the functions $d_s(\cdot, C_1)= d_s(\cdot,C_2)$.  Since the distance functions $d_s(\cdot, C_1)$ and
$d_s(\cdot,C_2)$ are distinct branches, they have no interior local maximum on $\Sigma_{12}$. In fact, at any (interior) critical point  on $\Sigma_{12}$ the gradients of 
$d_s(\cdot, C_1)$  and $d_s(\cdot,C_2)$ are collinear, hence the critical point  (if it exists) is the midpoint of the shortest geodesics joining $C_1$ and $C_2$. 
Thus $p\in \partial D'_{12}$ and there exists $C\in \mathcal C$ such that  $d_s(p, C)= d_s(p, C_1)= d_s(p,C_2)$. As a consequence there exists points $p'$ near $p$
such that $p' \in \Sigma_{12} \setminus \bar D'_{12}$, and moreover $d_s(p', C) < d_s(p, C_1)= d_s(p,C_2)$. In fact, such points can be found my following the gradient
lines of the function $d_s(\cdot, C)$ in the decreasing direction.  Such a contradiction concludes the proof.

\end{proof}

\noindent The topological structure of the cut-locus $\mathcal C(s)$ (for $s$ large) under the hypothesis that there are no periodic orbits is summarized by the notion introduced below.

\begin{definition}  \label{def:simple_branched_surface}
A {\it closed simple branched surface} is a type of branched surface defined as follows. Let $(\Sigma_0, \mathcal D, \sigma)$ be the data of an orientable closed surface $\Sigma_0$, a finite  cover $\mathcal D$ of $\Sigma_0$ by homotopically trivial closed subsets which are neighborhood deformation retracts in $\Sigma_0$, and $\sigma$ an involution on  $\mathcal D$ such that 
$$
\sigma (\bar D) \cap \bar D = \emptyset \quad \text{and} \quad  \bar D \cap  \bar D' \not= \emptyset \Longrightarrow  \sigma ( \bar D ) \cap  \sigma ( \bar D') =\emptyset\,,  \quad \text{ for all } \bar D, \bar D' \in \mathcal D \,.
$$
Let us assume that there exists a set  $\Phi_\sigma:= \{ \phi_{\bar D} \in \text{Homeo} (\bar D, \sigma(\bar D))\}$ of homeomorphism such that 
$$
\phi_{\sigma(\bar D)}  = \phi^{-1}_{\bar D}  \,, \quad \text{ for all } \bar D \in \mathcal D\,.
$$
Let $\approx_{\mathcal D}$ denote the equivalence relations
$$
x \approx_{\mathcal D} y \Longleftrightarrow  (x,y) = (x, \phi_{\bar D } (x) ) \in  \bar D \times \sigma(\bar D)\,, \quad \text{ for all }(x,y) \in \Sigma_0\times \Sigma_0\,.
$$
A closed simple branched surface  is a topological space defined as the  quotient space $\Sigma = \Sigma_0 / \approx_{\mathcal D}$ of the 
surface $\Sigma_0$ with respect to the equivalence relation  $\approx_{\mathcal D}$.
\end{definition}
\begin{rmk} By Definition \ref{def:simple_branched_surface}, any closed simple branched surface is Hausdorff since it is a quotient of a topological surface
with respect to a closed relation (as the cover $\mathcal D$ is given by closed subsets), and it is compact, since the surface $\Sigma_0$ is compact. The definition
does not require $\Sigma_0$ to be connected, hence a closed simple branched surface may or may not be connected. 
\end{rmk}

\smallskip
\noindent The real homology of a closed simple branched surface can be bounded below in terms of the cardinality of the defining cover $\mathcal D$, which defines it up
to homeomorphisms,  as it can be derived by induction from the following result.

\begin{lemma} \label{lemma:Betti_induction_step} Let $X$ be a path-connected Hausdorff topological space and let  ${\bar D}_1 ,  {\bar D}_2 \subset X$ be disjoint closed homeomorphic
subsets. Let $\approx$ be the closed equivalence relation in $X$ that identifies ${\bar D}_1$ and
 ${\bar D}_2$ by a homeomorphism $\phi_{12}: \bar D_1 \to \bar D_2$, that is 
$$
x_1 \approx x_2 \Longleftrightarrow  x_1=x_2 \quad \text{or} \quad  (x_1, x_2) \in {\bar D}_1 \times {\bar D}_2 \text{ and }  x_2 =\phi_{12} (x_1)\,,
$$
and let $\hat X = X/\approx$ be the quotient space. Let us assume that there exist open neighborhoods ${\bar D}_1 \subset \mathcal U_1 \subset X$ and
${\bar D}_2 \subset \mathcal U_2 \subset X$ such that  $\mathcal U_1$ and $\mathcal U_2$ are
both homotopically trivial in $X$. Then
$$
\text{ dim } H_1(\hat X, \R) \geq  \text{ dim } H_1(X, \R)  + 1 \,.
$$
\end{lemma}
\begin{proof} Since ${\bar D}_1$ and  ${\bar D}_2$ are closed and disjoint, it is not restrictive to assume that $ \mathcal U_1$ and $ \mathcal U_2$
are also disjoint. Let $A = X \setminus ({\bar D}_1 \times {\bar D}_2) \subset \hat X $ and let $B$ defined as
$$
B:= ( \mathcal U_1 \cup  \mathcal U_2 )/ \approx\,.
$$
By definition, since $ \mathcal U_1$ and  $\mathcal U_2 $ are open, we have that  the interiors $\text{int}(A)$ and $\text{int}(B)$ of
$A$ and $B$ form an open covering of $\hat X$. By the reduced Mayer-Vietoris exact sequence
$$
\begin{aligned}
\dots \rightarrow H_1(A\cap B, \R) &\rightarrow  H_1(A, \R) \oplus H_1(B, \R) \rightarrow H_1(\hat X, \R)  \\ &\rightarrow  \tilde{H}_0(A\cap B, \R) \rightarrow  \tilde{H}_0(A, \R) \oplus \tilde{H}_0(B, \R) \rightarrow \dots
\end{aligned}
$$
Since $A\cap B$ is homotopically trivial in $B$, the inclusion $H_1(A\cap B, \R) \to H_1(B, \R)$ is trivial. Let then  $\imath^{A\cap B,A}_*: H_1(A\cap B, \R) \to H_1(A, \R)$ the 
homology homomorphism induced by the inclusion $\imath^{A\cap B,A}: A\cap B \to A$. We have that 
$$
\text{dim } H_1(A, \R) \geq \text{dim } H_1(X, \R) +  \text{dim } \text{Im}(\imath^{A\cap B,A}_*)\,.
$$
In fact, since $A\cap B= (\mathcal U_1 \setminus \bar D_1) \sqcup (\mathcal U_2 \setminus \bar D_2)$, it follows that 
$$H_1(A\cap B, \R)  \approx H_1(\mathcal U_1 \setminus \bar D_1, \R) \oplus H_1(\mathcal U_2 \setminus \bar D_2, \R)$$ 
with  $H_1(\mathcal U_1 \setminus \bar D_1, \R)$ and $H_1(\mathcal U_1 \setminus \bar D_1, \R)$ both injecting trivially into $H_1(X,\R)$,
 since $\mathcal U_1$ and $\mathcal U_2$ are homotopically trivial in $X$. Alternatively we can reason as follows. By the reduced homology long exact sequence
$$
\begin{aligned}
\dots \to H_1(\mathcal U_1 \cup \mathcal U_2, \R) &\to  H_1(X, \R) \to H_1 (X, \mathcal U_1 \cup \mathcal U_2; \R) \\ &\to \tilde H_0(\mathcal U_1 \cup \mathcal U_2, \R) \to  \tilde H_0(X, \R)  \to \dots\,.
\end{aligned}
$$
Since $\mathcal U_1 \cup \mathcal U_2$ is a disjoint union of homotopically trivial sets, we have that the homology 
$H_1(\mathcal U_1 \cup \mathcal U_2, \R)=\{0\}$ and $\text{dim }\tilde H_0(\mathcal U_1 \cup \mathcal U_2, \R)=1$, since $X$ is path-connected we have that $ \tilde H_0(X, \R)=\{0\}$, hence 
$$
\text{dim }H_1 (X, \mathcal U_1 \cup \mathcal U_2; \R)= \text{dim }H_1 (X, \R) + 1\,.
$$
By the excision property 
$H_\ast (X, \mathcal U_1 \cup \mathcal U_2; \R)$ is isomorphic to $H_\ast (A, A\cap B; \R)$. We then examine the  reduced homology long exact sequence
$$
\dots \to H_1(A\cap B, \R) \to  H_1(A, \R) \to H_1 (A, A\cap B; \R) \to \tilde H_0(A\cap B, \R) \to \dots\,.
$$
Now we have that since $A\cap B = \mathcal U_1 \setminus \bar D_1 \sqcup \mathcal U_2 \setminus \bar D_2$, and $ \text{dim } \tilde H_0(A\cap B, \R)=1$, hence 
$$
\begin{aligned}
 \text{dim } H_1(A, \R) &\geq    \text{dim } \text{Im}(\imath^{A\cap B,A}_*) +  \text{dim } H_1 (A, A\cap B; \R) -1 \\ &=  \text{dim } H_1 (X, \R)+ \text{dim } \text{Im}(\imath^{A\cap B,A}_*)\,.
\end{aligned}
$$

Finally, since $A\cap B$ has two path-connected components, while $A$ and $B$ are path-connected, we have that 
$$
\text{dim } \tilde{H}_0(A\cap B, \R) = 1 \quad \text{ and } \quad    \text{dim } \tilde{H}_0(A, \R) =  \text{dim } \tilde{H}_0(B, \R)=0\,.
$$
The statement follows since by the Mayer-Vietoris exact sequence
$$
\begin{aligned}
\text{dim } H_1(\hat X , \R) &\geq \text{dim } \tilde{H}_0(A\cap B, \R) + \text{dim } H_1(A, \R)  \\ &-  \text{dim } \text{Im}(\imath^{A\cap B,A}_*) \geq \text{dim } H_1(X, \R) +1\,.
\end{aligned}
$$
\end{proof} 

\noindent The dimension of the homology of a simple branched surface can be bounded from below in terms the homology of the underlying topological surface and of the cardinality of the covering. More precisely we have
\begin{lemma} 
\label{lemma:Betti_bound}
Let $\Sigma$ be a closed simple branched surface given by data $(\Sigma_0, \mathcal D, \sigma)$ with $\Sigma_0$ a closed orientable surface, 
$\mathcal D$ a finite covering (of cardinality $\# \mathcal D$) of $\Sigma_0$ by homotopically trivial, closed, neighborhood deformation retracts and $\sigma$ 
an involution on the set $\mathcal D$.  Then
$$
\text{\rm dim } H_1(\Sigma, \R) \geq  \text{\rm dim } H_1(\Sigma_0, \R)  + \frac{1}{2} ( \# \mathcal D  )\,.
$$
\end{lemma} 
\begin{proof} It follows from Lemma \ref{lemma:Betti_induction_step} by induction on the cardinality $\# \mathcal D$ of the covering. 
In fact,  let $\bar D_0 \in \mathcal D$ and let $\mathcal D' =
\mathcal D \setminus \{ \bar D^*, \sigma (\bar D^*)\}$.  The data $(\Sigma_0, \mathcal D', \sigma\vert \mathcal D')$ defined a simple branched
surface $\Sigma'$ such that $\Sigma$ can be obtained from $\Sigma'$ by identifying the closed sets $\bar D^* \in \mathcal D$ and 
$\sigma (\bar D^*)$ via the homeomorphism $\phi_{\bar D^*}$.  Since $\bar D^*$ and $\sigma (\bar D^*)$ are homotopically trivial  neighborhood deformation 
retracts in $\Sigma_0$, by the definition of simple branched surface, in particular by the condition that $\sigma (\bar D) \cap \sigma (\bar D') =\emptyset$ 
whenever $\bar D \cap \bar D' \not = \emptyset$, for all $\bar D, \,,\bar D' \in \mathcal D$, the sets $\bar D^*$ and $\sigma (\bar D^*)$ are also homotopically trivial  
neighborhood deformation retracts in $\Sigma'$. Thus by Lemma \ref{lemma:Betti_induction_step} we have
$$
\text{dim} H_1(\Sigma, \R) \geq \text{dim} H_1(\Sigma', \R)  + 1 \,.
$$
Since $\# \mathcal D' =  \#\mathcal D -2$, the stated lower bound on $\text{dim} H_1(\Sigma, \R)$ follows from the induction hypothesis and the
argument is therefore complete.

\end{proof}

\begin{lemma} 
\label{lemma:simple_branched}
If the geodesic flow $\{ G_t\}$ has no periodic orbits,  then the cut-locus $\mathcal C(s)$ is homeomorphic to a simple  branched closed surface.
\end{lemma} 
\begin{proof}
By Lemma \ref{lemma:homotopy_trivial}, the connected components of $\mathcal C^*(s)$ have  homotopically trivial closure in $\mathcal C(s)$.
Let $H_s: \partial \hat S_r(M) \to \mathcal C(s)$ denote the map defined as follows: for each 
$x \in \partial \hat S_r (M)$, let $H_s(x) \in \mathcal C(s)$ be the endpoint in $\mathcal C(s)$ of the geodesics starting at $x$, normal to the boundary and ending in
$\mathcal C(s)$.  

Let $\mathcal C^*(s) = \bigcup_{\alpha \in A} \mathcal C_\alpha(s)$ the finite decomposition of $\mathcal C^*(s)$
 into connected components. For every $\alpha \in A$,  we consider the two disjoint open sets  $D_{\alpha,1}(s)$  and $D_{\alpha,2}(s)$ in  $\partial \hat S_r(M)$ such that 
$$
D_{\alpha,1}(s) \cup D_{\alpha,2}(s) = H_s^{-1} (\mathcal C_\alpha(s) )\,.
$$
The map  $H_s \vert \bar D_{\alpha,i}(s) \to  \mathcal C (s)$ is a homeomorphism from the closure  $\bar D_{\alpha,i}(s)$ of  $D_{\alpha,i}(s) $
onto the closure $\bar {\mathcal C}_\alpha(s)$  of $\mathcal C_\alpha(s)$ for $i\in \{1,2\}$, hence for all $\alpha \in A$ the map
$$
H_s \circ  (H_s \vert \bar D_{\alpha,i}(s) )^{-1} : \bar D_{\alpha,i}(s) \to \bar D_{\alpha,j}(s) \,, \quad \text{ for all } i\not=j \in \{1, 2\}.
$$
is a homeomorphism. We then define a simple branched closed surface, in the sense of Definition~\ref{def:simple_branched_surface} as follows: let the 
$\Sigma_0= \partial \hat S_r(M)$;  let the cover $$\mathcal D_s: = \{ \bar D_{\alpha,1}(s) \}_{\alpha \in A} \cup  \{ \bar D_{\alpha,2}(s) \}_{\alpha \in A} $$ 
and let the involution $\sigma_s: \mathcal D_s \to \mathcal D_s$ be defined as 
$$
\sigma_s ( \bar D_{\alpha,i}(s)) =  \bar D_{\alpha,j}(s)  \,, \quad \text{ for all } i\not=j \in \{1, 2\}\,.
$$
The surface $\Sigma_0$ is a closed orientable surface, by Lemma \ref{lemma:smooth} and Lemma \ref{lemma:homotopy_trivial} the sets of the cover $\mathcal D_s$ 
are closed, homotopically trivial neighborhood deformation retracts, since any closed homotopically trivial subset with piece-wise smooth boundary of a smooth
orientable surface is a neighborhood deformation retract; the map $\sigma_s: \mathcal D_s \to \mathcal D_s$ is an involution by definition and the homeomorphisms
$$
\phi_{ \bar D_{\alpha,i}(s)} := H_s \circ  (H_s \vert \bar D_{\alpha,1}(s))^{-1} :\bar D_{\alpha,i}(s) \to \bar D_{\alpha,j}(s)
 \,, \quad \text{ for all } i\not=j \in \{1, 2\}\,,
$$
are  such that  
$$
\phi_{ \sigma_s(\bar D_{\alpha,i}(s))} = \phi_{ \bar D_{\alpha,j}(s)}  =  \phi^{-1}_{ \bar D_{\alpha,i}(s)}   \,, \quad \text{ for all } i\not=j \in \{1, 2\}\,.
$$
Finally the homeomorphism between the cut-locus and the simple branched surface $\Sigma(s)$ given by the triple $(\Sigma_0, \mathcal D_s, \sigma_s)$
defined above is induced by the map $H_s : \Sigma_0 \to \mathcal C(s)$. In fact, the branched surface $\Sigma$ is precisely defined by the relation $\approx$
such that for any $x,y \in \Sigma_0 = \partial \hat S_r(M)$, 
$$
x \approx y \Longleftrightarrow  H_s(x) = H_s(y) \,.
$$
Since $H_s$ is surjective, the induced map is a bijection, and since $H_s$ is continuous the induced map is continuous. As the branched surface $\Sigma(s)$ is
given by the a closed relation, it is a Hausdorff space, which is also compact since the surface $\Sigma_0$ is compact, hence the induced map is a homeomorphism. 
The argument is complete. 

\end{proof}

\begin{lemma}
\label{lemma:Betti_div}
If the geodesic flow $\{ G_t\}$ has no periodic orbits, then the Betti number of the cut-locus $\mathcal C(s)$ diverges: 
$$
\lim_{s\to +\infty} \text{\rm dim } H_1(\mathcal C(s), \R) = +\infty\,.
$$
\end{lemma}
\begin{proof}  By Lemma \ref{lemma:simple_branched} the cut-locus $\mathcal C(s)$ is homeomorphic to a closed simple branched surface $\Sigma(s)$
defined as a quotient of the boundary surface $\Sigma_0 = \partial \hat S_r(M)$ with respect to the equivalence relation given by the closed covering 
$\mathcal D_s$ with elements the inverse images under the map $H_s: \Sigma_0 \to \mathcal C(s)$ of the  closure of the connected component of the regular 
subset $\mathcal C^*(s)$ of the cut-locus $\mathcal C(s)$. By Lemma \ref{lemma:Betti_bound} it is enough to prove that the cardinality of the cover $\mathcal D_s$ 
of the closed simple branched surface $\Sigma(s)$ diverges.

\smallskip
\noindent Let $\Gamma$ be the set of segments joining cone points. It can be proved that this set is infinite (countable). In fact, if it is finite
from every cone points there are only finitely many rays ending in a cone point. This implies that for any cone point there is an infinite  sector of given positive 
angle with vertex at the cone point which does not contain any cone points in its interior. This is not possible since it would imply that the surface contains 
euclidean disks of arbitrarily large radius, but every metric  disk of radius larger than the (finite) diameter of the surface contains all cone points.
Let $\Gamma^{(ij)} \subset \Gamma$ denote the subset of segments joining the cone points $p_i$ and~$p_j$. Since there are finitely many cone points, 
there exists a pair $(i,j)$ such that $\Gamma^{(ij)}$ is infinite. Let $\Gamma^{(ij)}_N=\{\gamma^{(ij)}_1, \dots, \gamma^{(ij)}_N\}$ denote the finite subset 
of $\Gamma^{(ij)}$ containing the first $N$ segments with respect to the total order given by their flat length. 

For $s$ large enough, for every $k\in \{1, \dots, N\}$,
the unit tangent vectors $v^{(ij)}_{k, \pm}$ based that the midpoint  $p^{(ij)}_k$ of $\gamma^{(ij)}_k$, parallel to the tangent vector of $\gamma^{(ij)}_k$, belong to 
the cut-locus $\mathcal C(s)$, and their minimum distance to the boundary is realized by the endpoints of the segment $\gamma^{(ij)}_{k, \pm}$ (which belong to the boundary 
components given by the cone points $p_i$ and $p_j$). In fact, all geodesics for the scaled metric $R_s$ starting at $v^{(ij)}_{k, \pm}$ of length smaller than  
$d_s(v^{(ij)}_{k, \pm}, \partial  \hat S_r(M))$ are contained in sector of aperture which converges to zero as $s$ diverges. Since, for $s$ large enough, the cone
points closest to $p^{(ij)}_k$ inside the sector are $p_i$ and $p_j$, it follows that the geodesics from $v^{(ij)}_{k, \pm}$ to $ \partial  \hat S_r(M)$ end in those points.

It remains to prove that the points $v^{(ij)}_{k, \pm}$ belong to distinct connected components of the regular subset $\mathcal C^*(s)$ of the cut-locus. We have
proved above that $v^{(ij)}_{k, \pm}\in \mathcal C^*(s)$ as the distance of $v^{(ij)}_{k, \pm}\in \mathcal C^*(s)$ is realized exactly by the radial unit tangent unit vectors
at the two conical points $p_i$ and $p_j$. If $v^{(ij)}_{k, \pm}$ and $v^{(ij)}_{k', \pm}$ belong to the same connected component of $\mathcal C^*(s)$, then the
paths $\gamma^{(ij)}_{k, \pm}$ and $\gamma^{(ij)}_{k', \pm}$ are homotopic within the set of paths of $R_s$-length less than $d_0$ joining the boundary components
$C(p_i)$ and $C(p_j)$ relative to the cone points $p_i$ and $p_j$. The projection of the image of the homotopy to the surface gives a connected region with no cone 
points in its interior and with the cone points $p_i$ and $p_j$ in its boundary. For $s>0$ large enough, any such connected region is contractible to a single segment
joining $p_i$ and $p_j$, therefore, taking into account the orientation, we conclude that $\gamma^{(ij)}_{k, \pm}=\gamma^{(ij)}_{k', \pm}$.

Thus the regular subset $\mathcal C^*(s)$ of the cut-locus has, for $s>0$ large enough, at least $N \in \N$ connected components, and since $N$ is arbitrary, we have proved
that the cardinality of the cover $\mathcal D_s$, which by Lemma \ref{lemma:simple_branched} defines the closed simple branched surface $\Sigma(s)$ homeomorphic to $\mathcal C(s)$, diverges, as claimed, hence by Lemma \ref{lemma:Betti_bound} the Betti number of $\mathcal C(s)$ diverges as well.

\end{proof} 

Finally, the following result implies that the topological complexity of the cut-locus is dominated by that of the manifold. 

Let $M$ be a smooth manifold with non-empty boundary.  The skeleton (or cut-locus) $\mathcal C$ of $M$ is defined 
as the set of points such that 
$$
\# \{y\in \partial M \vert d(x, \partial M) = d(x, y) \} > 1 \,.
$$
\begin{theorem} \cite{VEL78} 
\label{thm:rectraction} The closure $\overline{\mathcal C}$ of the cut-locus is a deformation retract of $M$, that is, there exists a continuous map $r: M \to \overline{\mathcal C}$
such that the restriction of $r$ to $\overline{\mathcal C}$ is the identity map of $\overline{\mathcal C}$ and which is homotopic to the identity map of $M$.

\end{theorem}

\begin{corollary}
\label{cor:Betti_bounded}
 The first homology of  the cut-locus $H_1(\mathcal C(s), \R)$ has dimension at most  $2g(M) + \# C_r $, with $g(M)$ the genus of the surface $M$ and 
$\# C_r$ equal to the number of conical points.
\end{corollary}  
\begin{proof} By Lemma \ref{lemma:cutlocus} the cut-locus $\mathcal C(s)$ is closed, and by Theorem~\ref{thm:rectraction} it is a retraction of $ \hat S_r(M)$. 
The retraction map $R:  \hat S_r(M) \to \mathcal C(s)$ induces a surjective map $R_* : H_1( \hat S_r(M), \R) \to  H_1( \mathcal C(s), \R)$. In fact,  let $j: \mathcal C(s) \to
 \hat S_r(M) $ denote the inclusion.  By definition of retraction $R \circ j = \text{Id}_{ \mathcal C(s)}$  (and $R$ is homotopic to $\text{Id}_{\hat S_r(M)}$).  It follows that 
 $R_* \circ j_* =  \text{Id}_{ \mathcal C(s)}$ hence $R_* : H^1( \hat S_r(M), \R) \to  H_1( \mathcal C(s), \R)$ has a right inverse and therefore is surjective.
 The dimension of the real vector space  $H^1( \hat S_r(M), \R)$ is at most equal to  $2g(M) + \# C_r $.  Indeed, the by the Gysin exact sequence the dimension of the first cohomology 
 can be computed as follows
 $$
 \text{ \rm dim } H^1( \hat S_r(M), \R) =   \text{ \rm dim } H^1(M\setminus C_r, \R) + 1= (2g(M) + \# C_r -1) + 1 \,.
 $$
 \end{proof} 
\noindent We are then ready to conclude.

\begin{proof}[Proof of Theorem \ref{thm:B}]  By Lemma \ref{lemma:Betti_div}  if the the geodesic flow $\{G_t\}$ has no periodic orbits, then the Betti number $\text{\rm dim } H_1( \mathcal C(s), \R)$ of the cut-locus  diverges. However, by Corollary \ref{cor:Betti_bounded} the Betti number is bounded in terms of the topology of the flat surface and of the cardinality of the set of is conical points. This contradiction proves that $\{G_t\}$ has at least one periodic orbit.

\end{proof}.

\bibliography{biblio}

@incollection {Sm99,
    AUTHOR = {Smillie, J.},
     TITLE = {The dynamics of billiard flows in rational polygons},
    BOOKTITLE = {Dynamical Systems, Ergodic Theory and Applications},
    SERIES = {Encyclopedia of Mathematical Sciences},
    VOLUME={100},
    EDITOR= {Ya. G. Sinai},
   PAGES = {360--380}, 
   PUBLISHER = {Springer Verlag},
   ADDRESS = {Berlin, Heidelberg},
      YEAR = {2000},}

@article {GKT,
    AUTHOR = {Galperin, G. and Kr\"uger, T. and Troubetzkoy, S.},
     TITLE = {Local Instability of Orbits in Polygonal and Polyhedral Billiards},
   JOURNAL = {Comm. Math. Phys.},
  FJOURNAL = {Communications in Mathematical Physics},
    VOLUME = {169},
      YEAR = {1995},
    NUMBER = {3},
     PAGES = {463--473},
}

@article {Fl92,
    AUTHOR = {Flaminio, L.},
     TITLE = {Une remarque sur les distributions invariantes par les flots
              g\'eod\'esiques des surfaces},
   JOURNAL = {C. R. Acad. Sci. Paris S\'er. I Math.},
  FJOURNAL = {Comptes Rendus de l'Acad\'emie des Sciences. S\'erie I.
              Math\'ematique},
    VOLUME = {315},
      YEAR = {1992},
    NUMBER = {6},
     PAGES = {735--738},
}

@incollection {Fo02b,
    AUTHOR = {Forni, G.},
     TITLE = {Asymptotic Behaviour of Ergodic Integrals of
                  Renormalisable Parabolic Flows},
 BOOKTITLE = {Proceedings of the ICM 2002},
    VOLUME = {III},
     PAGES = {317--326},
 PUBLISHER = {Higher Education Press},
   ADDRESS = {Beijing, China},
      YEAR = {2002},}

@book {F,
    AUTHOR = {Furstenberg, H.},
     TITLE = {Recurrence in Ergodic Theory and Combinatorial Number Theory },
    SERIES = {M. B. Porter lectures ; 1978},
 PUBLISHER = {Princeton Univ. Press},
   ADDRESS = {Princeton, NJ},
      YEAR = {1981},
     PAGES = {vii + 199},
     }

@article {GK80,
    AUTHOR = {Guillemin, V. and Kazhdan, D.},
     TITLE = {Some inverse spectral results for negatively curved
              {$2$}-manifolds},
   JOURNAL = {Topology},
  FJOURNAL = {Topology. An International Journal of Mathematics},
    VOLUME = {19},
      YEAR = {1980},
    NUMBER = {3},
     PAGES = {301--312},
}

@article {KMS,
    AUTHOR = {Kerckhoff, S. and Masur, H. and Smillie, J.},
     TITLE = {Ergodicity of billiard flows and quadratic differentials},
   JOURNAL = {Ann. of Math. (2)},
  FJOURNAL = {Annals of Mathematics. Second Series},
    VOLUME = {124},
      YEAR = {1986},
    NUMBER = {2},
     PAGES = {293--311},
     }

@book {ST,
    AUTHOR = {Singer, I. M. and Thorpe, J. A.},
     TITLE = {Lecture notes on elementary topology and geometry},
      NOTE = {Reprint of the 1967 edition,
              Undergraduate Texts in Mathematics},
 PUBLISHER = {Springer-Verlag},
   ADDRESS = {New York},
      YEAR = {1976},
     PAGES = {viii+232},
}

@article {Tr86,
    AUTHOR = {Troyanov, M.},
     TITLE = {Les surfaces euclidiennes \`a singularit\'es coniques},
   JOURNAL = {Enseign. Math. (2)},
  FJOURNAL = {L'Enseignement Math\'ematique. Revue Internationale. IIe
              S\'erie},
    VOLUME = {32},
      YEAR = {1986},
    NUMBER = {1-2},
     PAGES = {79--94},
}

@article{Trb05,
     author={Troubetzkoy, S.},
     title = {Periodic billiard orbits in right triangles},
     journal = {Annales de l'Institut Fourier},
     pages = {29--46},
     publisher = {Association des Annales de l{\textquoteright}institut Fourier},
     volume = {55},
     number = {1},
     year = {2005},
     doi = {10.5802/aif.2088},
     zbl = {1063.37022},
     mrnumber = {2141287},
     language = {en},
     url = {https://aif.centre-mersenne.org/articles/10.5802/aif.2088/}
}

@article {Ma86,
    AUTHOR = {Masur, H.},
     TITLE = {Closed trajectories for quadratic differentials with an application to billiards},
   JOURNAL = {Duke Math. J.},
  FJOURNAL = {Duke Mathematical Journal},
    VOLUME = {53},
      YEAR = {1986},
    NUMBER = {2},
     PAGES = {307--314},
     }

@article {Ve93,
    AUTHOR = {Veech, W. A.},
     TITLE = {Flat surfaces},
   JOURNAL = {Amer. J. Math.},
  FJOURNAL = {American Journal of Mathematics},
    VOLUME = {115},
      YEAR = {1993},
    NUMBER = {3},
     PAGES = {589--689},
}

@unpublished{Ka04, 
    AUTHOR = {Katok, A.},
     TITLE = {Five most resistant problems in dynamics},
    NOTE={MSRI-Evans Lecture, Berkeley.  27 September 2004.  See https://www.msri.org/workshops/267/schedules/1789 and
    http://akatok.s3-website-us-east-1.amazonaws.com/pub/5problems-expanded.pdf.
    (Expanded version of a Lecture at IMPAN,  see https://www.impan.pl/~biuletyn/arch/5problems.pdf)},
    }

@incollection {Thu98,
    AUTHOR = {Thurston, W. P.},
     TITLE = {Shapes of polyhedra and triangulations of the sphere},
 BOOKTITLE = {The Epstein birthday schrift},
    SERIES = {Geom. Topol. Monogr.},
    VOLUME = {1},
     PAGES = {511--549 (electronic)},
 PUBLISHER = {Geom. Topol. Publ., Coventry},
      YEAR = {1998},
}

@article{VEL78,
author = {Vainshtein, A. G.  and  Efremovich, V. A. and Loginov, E. A.},
title = {The skeleton of a {R}iemannian manifold with boundary},
journal = {Russian Mathematical Surveys},
volume = {33},
number = {3},
pages = {181--182},
doi = {10.1070/RM1978v033n03ABEH002488},
url = {https://dx.doi.org/10.1070/RM1978v033n03ABEH002488},
year = {1978},
}

@article{Fa,
author = {Fagnano, G.},
title = {Problemata quaedam ad methodum maximorum et minimorum spectantia},
journal = {Nova Acta Eruditorum},
volume = { },
number = {},
pages = {281--303},
year = {1775},
}

@article{Sc06,
author = {R. E. Schwartz},
title = {Obtuse Triangular Billiards {I}: Near the $(2, 3, 6)$ Triangle},
journal = {Experimental Mathematics},
volume = {15},
number = {2},
pages = {161--182},
year = {2006},
publisher = {Taylor \& Francis},
doi = {10.1080/10586458.2006.10128961},
URL = {https://doi.org/10.1080/10586458.2006.10128961},
eprint = {https://doi.org/10.1080/10586458.2006.10128961}}

@article{Sc08,
author = {Schwartz, R. E.},
title = {Obtuse Triangular Billiards {II}: One Hundred Degrees Worth of Periodic Trajectories},
journal = {Experimental Mathematics},
volume = {18},
number = {2},
pages = {137--171},
year = {2009},
publisher = {Taylor \& Francis},
doi = {10.1080/10586458.2009.10128891},
URL = {https://doi.org/10.1080/10586458.2009.10128891},
eprint = {https://doi.org/10.1080/10586458.2009.10128891}
}

@incollection {Sc_ICM,
    AUTHOR = {Schwartz, R. E.},
     TITLE = {Survey Lecture on billiards},
    BOOKTITLE = {International Congress of Mathematicians 2022, Vol. IV},
    editors ={Beliaev, D. and Smirnov, S. editors}, 
   PAGES = {2392--2429}, 
   PUBLISHER = {EMS Press},
   ADDRESS = {},
      YEAR = {2023},}

@unpublished {TGMM,
  AUTHOR = {Tokarsky, G. and Garber, J. and Marinov, B. and Moore, K.},
   TITLE = {One Hundred and Twelve Point Three Degree Theorem}, 
   NOTE = {preprint available at arXiv:1808.06667v1}, 
   PAGES = {1--30},
 YEAR = {2018},
 }

@article{Mil,
author = {J. Milnor},
title = {Curvatures of left invariant metrics on {L}ie groups},
journal = {Advances in Mathematics},
volume = {21},
number = {3},
pages = {293-329},
year = {1976},
issn = {0001-8708},
doi = {https://doi.org/10.1016/S0001-8708(76)80002-3},
url = {https://www.sciencedirect.com/science/article/pii/S0001870876800023},
}
\bibliographystyle{amsplain}
\end{document}